\documentclass[14pt]{amsart}

\usepackage{amsmath,amssymb,amsfonts,enumerate,amsthm, amscd}
\usepackage[french, ngerman, italian, english]{babel}
\usepackage[colorlinks=true]{hyperref}
\usepackage{paralist}
\usepackage{verbatim}
\usepackage{color}
\usepackage{xcolor}
\usepackage{tikz}
\usetikzlibrary{arrows,decorations.pathmorphing,backgrounds,positioning,fit}
\RequirePackage{pgfrcs}

\theoremstyle{plain}
\newtheorem{thm}{\bf Theorem}[section]

\newtheorem{prop}[thm]{\bf Proposition}
\newtheorem{lem}[thm]{\bf Lemma}
\newtheorem{cor}[thm]{\bf Corollary}

\theoremstyle{definition}

\theoremstyle{remark}
\newtheorem{rem}[thm]{\bf Remark}
\newtheorem{exam}[thm]{\bf Example}

\theoremstyle{example}

\def \CM{{\mathrm{CM}}}

\def \1{\mathbf 1}

\def \lk{\mathrm{link}}
\def \Del{\Delta}

\def \NN{\mathbb N}

\def \H{\mathcal H}

\def \G{\mathcal G}

\def \C{\mathcal C}

\begin{document}

\title[Cohen-Macaulay-ness in codimension for simplicial complexes and expansion functor]
{Cohen-Macaulay-ness in codimension for simplicial complexes and expansion functor}
\author{Rahim Rahmati-Asghar}



\begin{abstract}
In this paper we show that expansion of a Buchsbaum simplicial
complex is $\CM_t$, for an optimal integer $t\geq 1$. Also, by
imposing extra assumptions on a $\CM_t$ simplicial complex, we prove
that it can be obtained from a Buchsbaum complex.
\end{abstract}

\maketitle

\section*{Introduction}

Set $[n]:=\{x_1,\ldots,x_n\}$. Let $K$ be a field and
$S=K[x_1,\ldots,x_n]$, a polynomial ring over $K$. Let $\Del$ be a
simplicial complex over $[n]$. For an integer $t\geq 0$, Haghighi,
Yassemi and Zaare-Nahandi introduced the concept of $\CM_t$-ness
which is the pure version of simplicial complexes
\emph{Cohen-Macaulay in codimension $t$} studied in \cite{MiNoSw}. A
reason for the importance of $\CM_t$ simplicial complexes is that
they generalizes two notions for simplicial complexes: being
Cohen-Macaulay and Buchsbaum. In particular, by the results from
\cite{Re,Sh}, $\CM_0$ is the same as Cohen-Macaulayness and $\CM_1$
is identical with Buchsbaum property.

In \cite{HaYaZa1}, the authors described some combinatorial properties
of $\CM_t$ simplicial complexes and gave some characterizations of them
and generalized some results of \cite{Hi,Mi}. Then, in \cite{HaYaZa2},
they generalized a characterization of Cohen-Macaulay bipartite graphs
from \cite{HeHi} and \cite{CoNa} on unmixed Buchsbaum graphs.

Bayati and Herzog defined the expansion functor in the category of
finitely generated multigraded $S$-modules and studied some homological
behaviors of this functor (see \cite{BaHe}). The expansion functor helps
us to present other multigraded $S$-modules from a given finitely generated
multigraded $S$-module which may have some of algebraic properties of the
primary module. This allows to introduce new structures of a given multigraded
$S$-module with the same properties and especially to extend some homological
or algebraic results for larger classes (see for example \cite[Theorem 4.2]{BaHe}.
There are some combinatorial versions of expansion functor which we will recall in this paper.

The purpose of this paper is the study of behaviors of expansion
functor on $\CM_t$ complexes. We first recall some notations and
definitions of $\CM_t$ simplicial complexes in Section 1. In the
next section we describe the expansion functor in three contexts,
the expansion of a simplicial complex, the expansion of a simple
graph and the expansion of a monomial ideal. We show that there is a
close relationship between these three contexts. In Section 3 we
prove that the expansion of a $\CM_t$ complex $\Del$ with respect to
$\alpha$ is $\CM_{t+e-k+1}$ but it is not $\CM_{t+e-k}$ where
$e=\dim(\Del^\alpha)+1$ and $k$ is the minimum of the components of
$\alpha$ (see Theorem \ref{main}). In Section 4, we introduce a new
functor, called contraction, which acts in contrast to expansion
functor. As a main result of this section we show that if the
contraction of a $\CM_t$ complex is pure and all components of the
vector obtained from contraction are greater than or equal to $t$
then it is Buchsbaum (see Theorem \ref{contract,CM-t}). The section
is finished with a view towards the contraction of simple graphs.

\section{Preliminaries}

Let $t$ be a non-negative integer. We recall from \cite{HaYaZa1}
that a simplicial complex $\Del$ is called $\CM_t$ or
\emph{Cohen-Macaulay in codimension $t$} if it is pure and for every
face $F\in\Del$ with $\#(F)\geq t$, $\lk_\Del(F)$ is Cohen-Macaulay.
Every $\CM_t$ complex is also $\CM_r$ for all $r\geq t$. For $t<0$,
$\CM_t$ means $\CM_0$. The properties $\CM_0$ and $\CM_1$ are the
same as Cohen-Macaulay-ness and Buchsbaum-ness, respectively.

The link of a face $F$ in a simplicial complex $\Del$ is denoted by
$\lk_\Del(F)$ and is $$\lk_\Del(F)=\{G\in\Del: G\cap F=\emptyset,
G\cup F\in\Del\}.$$ The following lemma is useful for checking the
$\CM_t$ property of simplicial complexes:

\begin{lem}\label{CM-t eq}
(\cite[Lemma 2.3]{HaYaZa1}) Let $t\geq 1$ and let $\Del$ be a nonempty complex. Then $\Del$ is $\CM_t$ if and only if $\Del$ is pure and $\lk_\Del(v)$ is $\CM_{t-1}$ for every vertex $v\in\Del$.
\end{lem}

Let $\G=(V(\G),E(\G))$ be a simple graph with vertex set $V$ and edge set $E$. The \emph{independence complex} of $\G$ is the complex $\Del_\G$ with vertex set $V$ and with faces consisting of independent sets of vertices of $\G$. Thus $F$ is a face of $\Del_\G$ if and only if there is no edge of $\G$ joining any two
vertices of $F$.

The \emph{edge ideal} of a simple graph $\G$, denoted by $I(\G)$, is an ideal of $S$ generated by all squarefree monomials $x_ix_j$ with $x_ix_j\in E(\G)$.

A simple graph $\G$ is called $\CM_t$ if $\Del_\G$ is $\CM_t$ and it is called \emph{unmixed} if $\Del_\G$ is pure.

For a monomial ideal $I\subset S$, We denote by $G(I)$ the unique minimal set of monomial generators of $I$.

\section{The expansion functor in combinatorial and algebraic concepts}

In this section we define the expansion of a simplicial complex and recall the expansion of a simple graph from \cite{Sc} and the expansion of a monomial ideal from \cite{BaHe}. We show that these concepts are intimately related to each other.

(1) Let $\alpha=(k_1,\ldots,k_n)\in\NN^n$. For $F=\{x_{i_1},\ldots,x_{i_r}\}\subseteq \{x_1,\ldots,x_n\}$ define
$$F^\alpha=\{x_{i_11},\ldots,x_{i_1k_{i_1}},\ldots,x_{i_r1},\ldots,x_{i_rk_{i_r}}\}$$
as a subset of $[n]^\alpha:=\{x_{11},\ldots,x_{1k_1},\ldots,x_{n1},\ldots,x_{nk_n}\}$. $F^\alpha$ is called \emph{the expansion of $F$ with respect to $\alpha$.}

For a simplicial complex $\Del=\langle F_1,\ldots,F_r\rangle$ on $[n]$, we define \emph{the expansion of $\Del$ with respect to $\alpha$} as the simplicial complex
$$\Del^\alpha=\langle F^\alpha_1,\ldots,F^\alpha_r\rangle.$$

(2) The \emph{duplication} of a vertex $x_i$ of a simple graph $\G$ was first introduced by Schrijver \cite{Sc} and it means extending its vertex set $V(\G)$ by a new vertex $x'_i$ and replacing $E(\G)$ by
$$E(\G)\cup\{(e\backslash\{x_i\})\cup\{x'_i\}:x_i\in e\in E(\G)\}.$$
For the $n$-tuple $\alpha=(k_1,\ldots,k_n)\in\NN^n$, with positive integer entries, the \emph{expansion} of the simple graph $\G$ is denoted by $\G^\alpha$ and it is obtained from $\G$ by successively duplicating $k_i-1$ times every vertex $x_i$.

(3) In \cite{BaHe} Bayati and Herzog defined the expansion functor in the category of finitely generated multigraded $S$-modules and studied some homological behaviors of this functor. We recall the expansion functor defined by them only in the category of monomial ideals and refer the reader to \cite{BaHe} for more general case in the category of finitely generated multigraded $S$-modules.

Let $S^\alpha$ be a polynomial ring over $K$ in the variables
$$x_{11},\ldots,x_{1k_1},\ldots,x_{n1},\ldots,x_{nk_n}.$$
Whenever $I\subset S$ is a monomial ideal minimally generated by
$u_1,\ldots,u_r$, the expansion of $I$ with respect to $\alpha$ is
defined by
$$I^\alpha=\sum^r_{i=1}P^{\nu_1(u_i)}_1\ldots P^{\nu_n(u_i)}_n\subset S^\alpha$$
where $P_j=(x_{j1},\ldots,x_{jk_j})$ is a prime ideal of $S^\alpha$ and $\nu_j(u_i)$ is the exponent of $x_j$ in $u_i$.

It was shown in \cite{BaHe} that the expansion functor is exact and
so $(S/I)^\alpha=S^\alpha/I^\alpha$. In the following lemmas we
describe the relations between the above three concepts of expansion
functor.

\begin{lem}\label{epansion s-R}
For a simplicial complex $\Del$ we have $I^\alpha_\Del=I_{\Del^\alpha}$. In particular, $K[\Del]^\alpha=K[\Del^\alpha]$.
\end{lem}
\begin{proof}
Let $\Del=\langle F_1,\ldots,F_r\rangle$. Since $I_\Del=\bigcap^r_{i=1}P_{F^c_i}$, it follows from Lemma 1.1 in \cite{BaHe} that $I^\alpha_\Del=\bigcap^r_{i=1}P^\alpha_{F^c_i}$. The result is obtained by the fact that $P^\alpha_{F^c_i}=P_{(F^\alpha_i)^c}$.
\end{proof}

Let $u=x_{i_1}\ldots x_{i_t}\in S$ be a monomial and $\alpha=(k_1,\ldots,k_n)\in\NN^n$. We set $u^\alpha=G((u)^\alpha)$
and for a set $A$ of monomials in $S$, $A^\alpha$ is defined $$A^\alpha=\bigcup_{u\in A} u^\alpha.$$
One can easily obtain the following lemma.

\begin{lem}
Let $I\subset S$ be a monomial ideal and $\alpha\in\NN^n$. Then $G(I^\alpha)=G(I)^\alpha$.
\end{lem}

\begin{lem}
For a simple graph $\G$ on the vertex set $[n]$ and $\alpha\in\NN^n$ we have $I(\G^\alpha)=I(\G)^\alpha$.
\end{lem}
\begin{proof}
Let $\alpha=(k_1,\ldots,k_n)$ and $P_j=(x_{j1},\ldots,x_{jk_j})$. Then it follows from Lemma 11(ii,iii) of \cite{BaHe} that
$$I(\G^\alpha)=(x_{ir}x_{js}:x_ix_j\in E(\G), 1\leq r\leq k_i,1\leq s\leq k_j)=\sum_{x_ix_j\in E(\G)}P_iP_j$$
$$=\sum_{x_ix_j\in E(\G)}(x_i)^\alpha (x_j)^\alpha=(\sum_{x_ix_j\in E(\G)}(x_i)(x_j))^\alpha=I(\G)^\alpha.$$
\end{proof}

\section{The expansion of a $\CM_t$ complex}

The following proposition gives us some information about the
expansion of a simplicial complex which are useful in the proof of
the next results.

\begin{prop}\label{ex indepen}
Let $\Del$ be a simplicial complex and let $\alpha\in\NN^n$.
\begin{enumerate}[\upshape (i)]
  \item For all $i\leq\dim(\Del)$, there exists an epimorphism $\theta:\tilde{H}_{i}(\Del^\alpha;K)\rightarrow\tilde{H}_{i}(\Del;K)$.

In particular in this case
$$\tilde{H}_{i}(\Del^\alpha;K)/\ker(\theta)\cong\tilde{H}_{i}(\Del;K);$$
  \item For $F\in\Del^\alpha$ such that $F=G^\alpha$ for some $G\in\Del$, we have
$$\lk_{\Del^\alpha}(F)=(\lk_\Del (G))^\alpha;$$
  \item For $F\in\Del^\alpha$ such that $F\neq G^\alpha$ for every $G\in\Del$, we have
$$\lk_{\Del^\alpha}F=\langle U^\alpha\backslash F\rangle\ast \lk_{\Del^\alpha}U^\alpha$$
for some $U\in\Del$ with $F\subseteq U^\alpha$. Here $\ast$ means the join of two simplicial complexes.

In the third case, $\lk_{\Del^\alpha}F$ is a cone and so acyclic,
i.e., $\tilde{H}_i(\lk_{\Del^\alpha}F;K)=0$ for all $i>0$.

\end{enumerate}
\end{prop}
\begin{proof}
(i) Consider the map $\pi:[n]^\alpha\rightarrow [n]$ by $\pi(x_{ij})=x_i$ for all $i,j$. Let the simplicial map $\varphi:\Del^\alpha\rightarrow\Del$ be defined by $\varphi(\{x_{i_1j_1},\ldots,x_{i_qj_q}\})=\{\pi(x_{i_1j_1}),\ldots,\pi(x_{i_qj_q})\}=\{x_{i_1},\ldots,x_{i_q}\}$. Actually, $\varphi$ is an extension of $\pi$ to $\Del^\alpha$ by linearity. Define $\varphi_\#:\tilde{\C}_q(\Del^\alpha;K)\rightarrow\tilde{\C}_q(\Del;K)$, for each $q$, by
$$\varphi_\#([x_{i_0j_0},\ldots,x_{i_qj_q}])=\left\{
  \begin{array}{ll}
    0 & \mbox{if for some indices}\ i_r=i_t \\
    \left[\varphi(\{x_{i_0j_0}\}),\ldots,\varphi(\{x_{i_qj_q}\})\right] & \mbox{otherwise}.
  \end{array}
\right.
$$
It is clear from the definitions of $\tilde{\C}_q(\Del^\alpha;K)$ and $\tilde{\C}_q(\Del;K)$ that $\varphi_\#$ is well-defined. Also, define $\varphi_\alpha:\tilde{H}_{i}(\Del^\alpha;K)\rightarrow\tilde{H}_{i}(\Del;K)$ by
$$\varphi_\alpha:z+B_i(\Del^\alpha)\rightarrow \varphi_\#(z)+B_i(\Del).$$
It is trivial that $\varphi_\alpha$ is onto.

(ii) The inclusion $\lk_{\Del^\alpha}(F)\supseteq(\lk_\Del (G))^\alpha$ is trivial. So we show the reverse inclusion. Let $\sigma\in\lk_{\Del^\alpha}(G^\alpha)$. Then $\sigma\cap G^\alpha=\emptyset$ and $\sigma\cup G^\alpha\in\Del^\alpha$. We want to show $\pi(\sigma)\in\lk_\Del (G)$. Because in this case, $\pi(\sigma)^\alpha\in(\lk_\Del (G))^\alpha$ and since that $\sigma\subseteq \pi(\sigma)^\alpha$, we can conclude that $\sigma\in (\lk_\Del (G))^\alpha$.

Clearly, $\pi(\sigma)\cup G\in\Del$. To show that $\pi(\sigma)\cap G=\emptyset$, suppose, on the contrary, that $x_i\in \pi(\sigma)\cap G$. Then $x_{ij}\in \sigma$ for some $j$. Especially, $x_{ij}\in G^\alpha$. Therefore $\sigma\cap G^\alpha\neq\emptyset$, a contradiction.

(iii) Let $\tau\in\lk_{\Del^\alpha}F$. Let $\tau\cap \pi(F)^\alpha=\emptyset$. It follows from $\tau\cup F\in \Del^\alpha$ that $\pi(\tau)^\alpha\cup\pi(F)^\alpha\in \Del^\alpha$. Now by $\tau\subset\pi(\tau)^\alpha$ it follows that $\tau\cup\pi(F)^\alpha\in \Del^\alpha$. Hence $\tau\in\lk_{\Del^\alpha}(\pi(F)^\alpha)$. So we suppose that $\tau\cap \pi(F)^\alpha\neq\emptyset$. We write
$\tau=(\tau\cap \pi(F)^\alpha)\cup (\tau\backslash \pi(F)^\alpha)$. It is clear that $\tau\cap \pi(F)^\alpha\subset \pi(F)^\alpha\backslash F$ and $\tau\backslash \pi(F)^\alpha\in\lk_{\Del^\alpha}\pi(F)^\alpha$. The reverse inclusion is trivial.
\end{proof}

\begin{rem}\label{pure expan}
Let $\Del=\langle x_1x_2,x_2x_3\rangle$ be a complex on $[3]$ and
$\alpha=(2,1,1)\in\NN^3$. Then $\Del^\alpha=\langle
x_{11}x_{12}x_{21},x_{21}x_{31}\rangle$ is a complex on
$\{x_{11},x_{12},x_{21},x_{31}\}$. Notice that $\Del$ is pure but
$\Del^\alpha$ is not. Therefore, the expansion of a pure simplicial
complex is not necessarily pure.
\end{rem}

\begin{thm}\label{main}
Let $\Del$ be a simplicial complex on $[n]$ of dimension $d-1$ and
let $t\geq 0$ be the least integer that $\Del$ is $\CM_t$. Suppose
that $\alpha=(k_1,\ldots,k_n)\in\NN^n$ such that $k_i>1$ for some
$i$ and $\Del^\alpha$ is pure. Then $\Del^\alpha$ is $\CM_{t+e-k+1}$
but it is not $\CM_{t+e-k}$, where $e=\dim(\Del^\alpha)+1$ and
$k=\min\{k_i:k_i>1\}$ .
\end{thm}
\begin{proof}
We use induction on $e\geq 2$. If $e=2$, then $\dim(\Del^\alpha)=1$ and $\Del$ should be only in form $\Del=\langle x_1,\ldots,x_n\rangle$. In particular, $\Del^\alpha$ is of the form
$$\Del^\alpha=\langle \{x_{i_11},x_{i_12}\},\{x_{i_21},x_{i_22}\},\ldots,\{x_{i_r1},x_{i_r2}\}\rangle.$$
It is clear that $\Del^\alpha$ is $\CM_1$ but it is not Cohen-Macaulay.

Assume that $e>2$. Let $\{x_{ij}\}\in\Del^\alpha$. We want to show that $\lk_{\Del^\alpha}(x_{ij})$ is $\CM_{e-k}$. Consider the following cases:

Case 1: $k_i>1$. Then
$$\lk_{\Del^\alpha}(x_{ij})=\langle\{x_i\}^\alpha\backslash x_{ij}\rangle\ast(\lk_{\Del}(x_i))^\alpha.$$
$(\lk_{\Del}(x_i))^\alpha$ is of dimension $e-k_i-1$ and, by induction hypothesis, it is $\CM_{t+e-k_i-k+1}$. On the other hand, $\langle\{x_i\}^\alpha\backslash x_{ij}\rangle$ is Cohen-Macaulay of dimension $k_i-2$. Therefore, it follows from Theorem 1.1(i) of \cite{HaYaZa2} that $\lk_{\Del^\alpha}(x_{ij})$ is $\CM_{t+e-k}$.

Case 2: $k_i=1$. Then
$$\lk_{\Del^\alpha}(x_{ij})=(\lk_{\Del}(x_i))^\alpha$$
which is of dimension $e-2$ and, by induction, it is $\CM_{t+e-k}$.

Now suppose that $e>2$ and $k_s=k$ for some $s\in [n]$. Let $F$ be a facet of $\Del$ such that $x_s$ belongs to $F$.

If $\dim(\Del)=0$, then $k_l=k$ for all $l\in [n]$. In particular, $e=k$. It is clear that $\Del^\alpha$ is not $\CM_{t+e-k}$ (or Cohen-Macaulay). So suppose that $\dim(\Del)>0$. Choose $x_i\in F\backslash x_s$. Then
$$\lk_{\Del^\alpha}(x_{ij})=\langle \{x_i\}^\alpha\backslash x_{ij}\rangle\ast (\lk_\Del(x_i))^\alpha.$$
By induction hypothesis, $(\lk_\Del(x_i))^\alpha$ is not $\CM_{t+e-k_i-k}$. It follows from Theorem 3.1(ii) of \cite{HaYaZa2} that $\lk_{\Del^\alpha}(x_{ij})$ is not $\CM_{t+e-k-1}$. Therefore $\Del^\alpha$ is not $\CM_{t+e-k}$.
\end{proof}

\begin{cor}
Let $\Del$ be a non-empty Cohen-Macaulay simplicial complex on
$[n]$. Then for any $\alpha\in\NN^n$, with $\alpha\neq\1$,
$\Del^\alpha$ can never be Cohen-Macaulay.
\end{cor}

\section{The contraction functor}

Let $\Del=\langle F_1,\ldots,F_r\rangle$ be a simplicial complex on $[n]$. Consider the equivalence relation `$\sim$' on the vertices of $\Del$ given by
$$x_i\sim x_j\Leftrightarrow\langle x_i\rangle\ast\lk_\Del(x_i)=\langle x_j\rangle\ast\lk_\Del(x_j).$$
In fact $\langle x_i\rangle\ast\lk_\Del(x_i)$ is the cone over
$\lk_\Del(x_i)$, and the elements of
$\langle x_i\rangle\ast\lk_\Del(x_i)$ are those faces of $\Delta$,
which contain $x_i$. Hence  $\langle x_i\rangle\ast\lk_\Del(x_i)
=\langle x_j\rangle\ast\lk_\Del(x_j)$, means the cone with vertex $x_i$ is equal to the cone with vertex $x_j$. In other words,
$x_i\sim x_j$ is equivalent to saying that for a facet $F\in\Delta$, $F$ contains $x_i$ if and only if it contains $x_j$.

Let $[\bar{m}]=\{\bar{y}_1,\ldots,\bar{y}_m\}$ be the set of
equivalence classes under $\sim$. Let
$\bar{y}_i=\{x_{i1},\ldots,x_{ia_i}\}$. Set
$\alpha=(a_1,\ldots,a_m)$. For $F_t\in\Del$, define
$G_t=\{\bar{y}_i:\bar{y}_i\subset F_t\}$ and let $\Gamma$ be a
simplicial complex on the vertex set $[m]$ with facets
$G_1,\ldots,G_r$. We call $\Gamma$  the \emph{contraction of $\Del$
by $\alpha$} and $\alpha$ is called \emph{the vector obtained from
contraction}.

For example, consider the simplicial complex $\Del=\langle x_1x_2x_3,x_2x_3x_4,x_1x_4x_5,x_2x_3x_5\rangle$
on the vertex set $[5]=\{x_1,\ldots,x_5\}$. Then $\bar{y}_1=\{x_1\}$, $\bar{y}_2=\{x_2,x_3\}$, $\bar{y}_3=\{x_4\}$,
$\bar{y}_4=\{x_5\}$ and $\alpha=(1,2,1,1)$. Therefore, the contraction of $\Del$ by $\alpha$ is
$\Gamma=\langle \bar{y}_1\bar{y}_2,\bar{y}_2\bar{y}_3,\bar{y}_1\bar{y}_3\bar{y}_4,\bar{y}_2\bar{y}_4\rangle$ a
complex on the vertex set $[\bar{4}]=\{\bar{y}_1,\ldots,\bar{y}_4\}$.

\begin{rem}
Note that if $\Del$ is a pure simplicial complex then the contraction of $\Del$ is not necessarily pure (see the above example). In the special case where the vector $\alpha=(k_1,\dots,k_n)\in\NN^n$ and $k_i=k_j$ for all $i,j$, it is easy to check that in this case $\Del$ is pure if and only if $\Del^\alpha$ is pure. Another case is introduced in the following proposition.
\end{rem}

\begin{prop}
Let $\Del$ be a simplicial complex on $[n]$ and assume that
$\alpha=(k_1,\dots,k_n)\in\NN^n$ satisfies the following condition:

$(\dag)$ for all facets $F,G\in\Del$, if $x_i\in F\backslash G$ and $x_j\in G\backslash F$ then $k_i=k_j$.

Then $\Del$ is pure if and only if $\Del^\alpha$ is pure.
\end{prop}
\begin{proof}
Let $\Del$ be a pure simplicial complex and let $F,G\in\Del$ be two facets of $\Del$. Then $$|F^\alpha|-|G^\alpha|=\sum_{x_i\in F}k_i-\sum_{x_i\in G}k_i=\sum_{x_i\in F\backslash G}k_i-\sum_{x_i\in G\backslash F}k_i.$$
Now the condition $(\dag)$ implies that $|F^\alpha|=|G^\alpha|$. This means that all facets of $\Del^\alpha$ have the same cardinality.

Let $\Del^\alpha$ be pure. Suppose that $F,G$ are two facets in
$\Del$. If $|F|>|G|$ then $|F\backslash G|>|G\backslash F|$.
Therefore $\sum_{x_i\in F\backslash G}k_i>\sum_{x_i\in G\backslash
F}k_i$. This concludes that $|F^\alpha|=\sum_{x_i\in
F}k_i>\sum_{x_i\in G}k_i=|G^\alpha|$, a contradiction.
\end{proof}

There is a close relationship between a simplicial complex and its contraction. In fact, the expansion of the contraction of a simplicial complex is the same complex. The precise statement is the following.

\begin{lem}
Let $\Gamma$ be the contraction of $\Del$ by $\alpha$. Then $\Gamma^\alpha\cong \Del$.
\end{lem}
\begin{proof}
Suppose that $\Del$ and $\Gamma$ are on the vertex sets $[n]=\{x_1,\ldots,x_n\}$ and $[\bar{m}]=\{\bar{y}_1,\ldots,\bar{y}_m\}$, respectively. Let $\alpha=(a_1,\ldots,a_m)$. For $\bar{y}_i\in\Gamma$, suppose that $\{\bar{y}_i\}^\alpha=\{\bar{y}_{i1},\ldots,\bar{y}_{ia_i}\}$. So $\Gamma^\alpha$ is a simplicial complex on the vertex set $[\bar{m}]^\alpha=\{\bar{y}_{ij}:i=1,\ldots,m,\ j=1,\ldots,a_i\}$. Now define $\varphi:[\bar{m}]^\alpha\rightarrow [n]$ by $\varphi(\bar{y}_{ij})=x_{ij}$. Extending $\varphi$, we obtain the isomorphism $\varphi:\Gamma^\alpha\rightarrow \Del$.
\end{proof}

\begin{prop}\label{CM indepen}
Let $\Del$ be a simplicial complex and assume that $\Del^\alpha$ is
Cohen-Macaulay for some $\alpha\in\NN^n$. Then $\Del$ is
Cohen-Macaulay.
\end{prop}
\begin{proof}
By Lemma \ref{ex indepen}(i), for all $i\leq \dim(\lk_\Del F)$ and all $F\in\Del$ there exists an epimorphism $\theta:\lk_{\Del^\alpha} F^\alpha\rightarrow \lk_\Del F$ such that
$$\tilde{H}_{i}(\lk_{\Del^\alpha} F^\alpha;K)/\ker(\theta)\cong\tilde{H}_{i}(\lk_\Del F;K).$$
Now suppose that $i<\dim(\lk_\Del F)$. Then $i<\dim(\lk_{\Del^\alpha} G^\alpha)$ and by Cohen-Macaulayness of $\Del^\alpha$, $\tilde{H}_{i}(\lk_{\Del^\alpha} F^\alpha;K)=0$. Therefor $\tilde{H}_{i}(\lk_\Del F;K)=0$. This means that $\Del$ is Cohen-Macaulay.
\end{proof}

It follows from Proposition \ref{CM indepen} that:

\begin{cor}\label{CM-ness}
The contraction of a Cohen-Macaulay simplicial complex $\Del$ is Cohen-Macaulay.
\end{cor}

This can be generalized in the following theorem.

\begin{thm}\label{contract,CM-t}
Let $\Gamma$ be the contraction of a $\CM_t$ simplicial complex
$\Del$, for some $t\geq 0$, by $\alpha=(k_1,\ldots,k_n)$. If
$k_i\geq t$ for all $i$ and $\Gamma$ is pure, then $\Gamma$ is
Buchsbaum.
\end{thm}
\begin{proof}
If $t=0$, then we saw in Corollary \ref{CM-ness} that $\Gamma$ is Cohen-Macaulay and so it is $\CM_t$. Hence assume that $t>0$. Let $\Del=\langle F_1,\ldots,F_r\rangle$. We have to show that $\tilde{H}_i(\lk_\Gamma G;K)=0$, for all faces $G\in\Gamma$ with $|G|\geq 1$ and all $i<\dim(\lk_\Gamma G)$.

Let $G\in\Gamma$ with $|G|\geq 1$. Then $|G^\alpha|\geq t$. It follows from Lemma \ref{CM-t eq} and $\CM_t$-ness of $\Del$ that
$$\tilde{H}_{i}(\lk_\Gamma G;K)\cong\tilde{H}_{i}(\lk_\Del G^\alpha;K)=0$$
for $i<\dim(\lk_\Del G^\alpha)$ and, particularly, for $i<\dim(\lk_\Gamma G)$. Therefore $\Gamma$ is Buchsbaum.
\end{proof}

\begin{cor}
Let $\Gamma$ be the contraction of a Buchsbaum simplicial complex
$\Del$. If $\Gamma$ is pure, then $\Gamma$ is also Buchsbaum.
\end{cor}

Let $\G$ be a simple graph on the vertex set $[n]$ and let $\Del_\G$ be its independence complex on $[n]$, i.e., a simplicial complex whose faces are the independent vertex sets of $G$. Let $\Gamma$ be the contraction of $\Del_\G$. In the following we show that $\Gamma$ is the independence complex of a simple graph $\H$. We call $\H$ the \emph{contraction} of $\G$.

\begin{lem}
Let $\G$ be a simple graph. The contraction of $\Del_\G$ is the independence complex of a simple graph $\H$.
\end{lem}
\begin{proof}
It suffices to show that $I_\Gamma$ is a squarefree monomial ideal generated in degree 2. Let $\Gamma$ be the contraction of $\Del_\G$ and let $\alpha=(k_1,\ldots,k_n)$ be the vector obtained from the contraction. Let $[n]=\{x_1,\ldots,x_n\}$ be the vertex set of $\Gamma$. Suppose that $u=x_{i_1}\ldots x_{i_t}\in G(I_\Gamma)$. Then $u^\alpha\subset G(I_\Gamma)^\alpha=G(I_{\Del_\G})=G(I(\G)$. Since $u^\alpha=\{x_{i_1j_1}\ldots x_{i_tj_t}:1\leq j_l\leq k_{i_l},1\leq l\leq t\}$ we have $t=2$ and the proof is completed.
\end{proof}

\begin{exam}
Let $\G_1$ and $\G_2$ be, respectively, from left to right the following graphs:

$$\begin{array}{cccc}
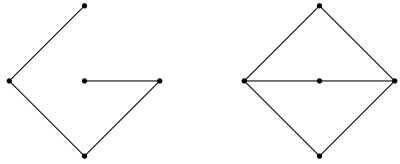
\begin{tikzpicture}
\coordinate (a) at (0,0);\fill (0,0) circle (1pt);
\coordinate (b) at (0,1);\fill (0,1) circle (1pt);
\coordinate (c) at (1,0);\fill (1,0) circle (1pt);
\coordinate (d) at (0,-1);\fill (0,-1) circle (1pt);
\coordinate (e) at (-1,0);\fill (-1,0) circle (1pt);
\draw[black] (a) -- (c) -- (d) -- (e) -- (b);
\end{tikzpicture}
&&&
\begin{tikzpicture}
\coordinate (a) at (0,0);\fill (0,0) circle (1pt);
\coordinate (b) at (0,1);\fill (0,1) circle (1pt);
\coordinate (c) at (1,0);\fill (1,0) circle (1pt);
\coordinate (d) at (0,-1);\fill (0,-1) circle (1pt);
\coordinate (e) at (-1,0);\fill (-1,0) circle (1pt);
\draw[black] (e) -- (a) -- (c) -- (d) -- (e) -- (b) -- (c);
\end{tikzpicture}
\end{array}$$

The contraction of $\G_1$ and $\G_2$ are

$$\begin{array}{cccc}
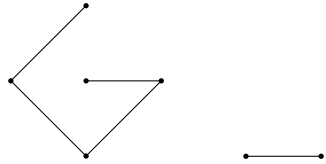
\begin{tikzpicture}
\coordinate (a) at (0,0);\fill (0,0) circle (1pt);
\coordinate (b) at (0,1);\fill (0,1) circle (1pt);
\coordinate (c) at (1,0);\fill (1,0) circle (1pt);
\coordinate (d) at (0,-1);\fill (0,-1) circle (1pt);
\coordinate (e) at (-1,0);\fill (-1,0) circle (1pt);
\draw[black] (a) -- (c) -- (d) -- (e) -- (b);
\end{tikzpicture}
&&&
\begin{tikzpicture}
\coordinate (a) at (0,0);\fill (0,0) circle (1pt);
\coordinate (b) at (1,0);\fill (1,0) circle (1pt);
\draw[black] (a) -- (b);
\end{tikzpicture}
\end{array}$$
The contraction of $\G_1$ is equal to itself but $\G_2$ is contracted to an edge and the vector obtained from contraction is $\alpha=(2,3)$.
\end{exam}

We recall that a simple graph is $\CM_t$ for some $t\geq 0$, if the associated independence complex is $\CM_t$.

\begin{rem}
The simple graph $\G'$ obtained from $\G$ in Lemma 4.3 and Theorem
4.4 of \cite{HaYaZa2} is the expansion of $\G$. Actually, suppose
that $\G$ is a bipartite graph on the vertex set $V(\G)=V\cup W$
where $V=\{x_1,\ldots,x_d\}$ and $W=\{x_{d+1},\ldots,x_{2d}\}$. Then
for $\alpha=(n_1,\ldots,n_d,n_1,\ldots,n_d)$ we have
$\G'=\G^\alpha$. It follows from Theorem \ref{main} that if $\G$ is
$\CM_t$ for some $t\geq 0$ then $\G'$ is $\CM_{t+n-n_{i_0}+1}$ where
$n=\sum^d_{i=1}n_i$ and $n_{i_0}=\min\{n_i>1:i=1,\ldots,d\}$. This
implies that the first part of Theorem 4.4 of \cite{HaYaZa2} is an
obvious consequence of Theorem \ref{main} for $t=0$.
\end{rem}

\subsection*{Acknowledgment}

The author would like to thank Hassan Haghighi from K. N. Toosi University of Technology and Rahim Zaare-Nahandi from University of Tehran for careful
reading an earlier version of this article and for their helpful comments.

\ \\ \\
Rahim Rahmati-Asghar,\\
Department of Mathematics, Faculty of Basic Sciences,\\
University of Maragheh, P. O. Box 55181-83111, Maragheh, Iran.\\
E-mail:  \email{rahmatiasghar.r@gmail.ac.ir}


\begin{thebibliography}{1}

\bibitem{BaHe} S. Bayati, J. Herzog, \textit{Expansions of monomial ideals and multigraded modules}, to appear in Rocky Mountain J. Math.

\bibitem{CoNa} D. Cook and U. Nagel, \textit{Cohen-Macaulay graphs and face vectors of flag complexes}, SIAM J. Discrete Math. \textbf{26} (2012), no. 1, 89-101.

\bibitem{HaYaZa1} H. Haghighi, S. Yassemi, R. Zaare-Nahandi, \textit{A generalization of $k$-Cohen-Macaulay simplicial complexes}, Ark. Mat. \textbf{50} (2012), 279-290.

\bibitem{HaYaZa2} H. Haghighi, S. Yassemi, R. Zaare-Nahandi, \textit{Cohen-Macaulay-ness in codimension for bipartite graphs}, to appear in Proc. Amer. Math. Soc.

\bibitem{HeHi} J. Herzog and T. Hibi, \textit{Distibutive lattices, bipartite graphs and Alexander duality}, J. Algebra Combin. \textbf{22} (2005), 289-302.

\bibitem{Hi} T. Hibi, \textit{Level rings and algebras with straightening laws}, J. Algebra, \textbf{117} (1988), 343-362.

\bibitem{MiNoSw} E. Miller, I. Novik and E. Swartz, \textit{Face rings of simplicial complexes with singularities}, Math. Ann. (2011), 351:857-875.

\bibitem{Mi} M. Miyazaki, \textit{Characterizations of Buchsbaum complexes}, Manuscripta Math. \textbf{63} (1989), 245-254.

\bibitem{Re} G.A. Reisner, \textit{Cohen-Macaulay quotients of polynomial rings}, Adv. Math. \textbf{21} (1976), 30-49.

\bibitem{Sc} A. Schrijver, \textit{Combinatorial Optimization}, Algorithms and Combinatorics \textbf{24}, Springer-Verlag, Berlin, 2003.

\bibitem{Sh} P. Schenzel, \textit{On the number of faces of simplicial complexes and the purity of Frobenius}, Math. Z. \textbf{178} (1981), 125-142.

\end{thebibliography}
\end{document}